\documentclass{elsarticle}

\usepackage{amsfonts,amssymb,latexsym,xspace,epsfig,graphics,color}
\usepackage{amsmath,enumerate,stmaryrd,xy}
\usepackage{bbm}
\usepackage{natbib}
\usepackage{subfigure}
\usepackage{dsfont}
\usepackage{a4wide}

\usepackage{tikz}
\usetikzlibrary{arrows,snakes,backgrounds}

\newtheorem{theorem}{Theorem}[section]                                          
\newtheorem{proposition}[theorem]{Proposition}                          
\newtheorem{lemma}[theorem]{Lemma}

\newtheorem{definition}[theorem]{Definition}
\newtheorem{remark}{Remark}[section]
\newtheorem{proof}{Proof}

\def\T+{{\mathbb T_d^+}}

\def\o{{\bf 0}}

\newcommand{\cam}[2]{\ensuremath{#1 \stackrel {\textit {\scriptsize {a.p.}}}{\longrightarrow} #2}}

\bibdata{prob}
\bibstyle{alpha} 

\begin{document}

\begin{frontmatter}

\title{On the existence of accessibility in a tree-indexed percolation model}

\author[CFC]{Cristian F. Coletti}
\ead{cristian.coletti@ufabc.edu.br}
\address[CFC]{Centro de Matem\'atica, Computa\c{c}\~ao e Cogni\c{c}\~ao - Universidade Federal do ABC (CMCC-UFABC),\\ Av. dos Estados, 5001, Bangu, Santo Andr\'e, SP, Brasil}

\author[rjg]{Renato J. Gava}
\ead{gava@ufscar.br}
\address[rjg]{Departamento de Estat\'istica, Universidade Federal de S\~ao Carlos (DEs-UFSCar),\\ Rodovia Washington Luiz, km 235, CEP 13565-905, S\~ao Carlos, SP, Brasil}

\author[PMR]{Pablo M. Rodr\'iguez\corref{cor1}}
\ead{pablor@icmc.usp.br}
\address[PMR]{Instituto de Ci\^encias Matem\'aticas e de Computa\c{c}\~ao, Universidade de S\~ao Paulo (ICMC-USP),\\ Av. Trabalhador s\~ao-carlense 400, Centro, CEP 13560-970, S\~ao Carlos, SP, Brasil}

\cortext[cor1]{Corresponding author}

\begin{abstract}
We study the accessibility percolation model on infinite trees. The model is defined by associating an absolute continuous random variable $X_v$ to each vertex $v$ of the tree. The main question to be considered is the existence or not of an infinite path of nearest neighbors $v_1,v_2,v_3\ldots$ such that 
$X_{v_1}<X_{v_2}<X_{v_3}<\cdots$ and which spans the entire graph. The event defined by the  existence of such path is called {\it{percolation}}.

We consider the case of the accessibility percolation model on a spherically symmetric tree with growth function given by $f(i)=\lceil (i+1)^ \alpha \rceil$, where $\alpha>0$ is a given constant. We show that there is a percolation threshold at $\alpha_c =1$ such that there is percolation if $\alpha> 1$ and there is absence of percolation if $\alpha \leq 1$. Moreover, we study the event of percolation starting at any vertex, as well as the continuity of the percolation probability function. Finally, we provide a comparison between this model with the well known $F^{\alpha}$ record model. We also discuss a number of open problems concerning the accessibility percolation model for further consideration in future research.
\end{abstract}

\begin{keyword}
Accessibility Percolation\sep Phase Transition
\end{keyword}

\end{frontmatter}
  



\section{Introduction} \label{intro}


Many percolation models were inspired by physics and developed in order to answer questions of physical and mathematical interest. This summarizes, roughly speaking, the beginning of
a modern theory with interesting rigorous results and a wide range of applications. The standard model of percolation appears as a mathematical model for the first time in the work of
Broadbent and Hammersley in 1957 --- see \cite{BH}. The main purpose of that work was to study how random properties of a porous medium influence the transport of a fluid through it. In order to
accomplish it, a lattice is used to represent the medium, where vertices are associated to pores and bonds to channels, and a family of Bernoulli independent and identically
distributed random variables with parameter $p$ is used to represent when a given channel is open or not to the spread of a fluid. The first question to be considered is the existence
or not of an infinite component of open channels and how this event depends on $p$. The answer gives rise to a result known as phase transition behavior, which guarantees that there exists a nontrivial critical value of the parameter $p$ called $p_c$ such that if $p$ is above $p_c$ the origin belongs to an infinite connected component with positive probability. The reader can find more details about mathematical formulation and main results for the standard percolation
model and related spatial processes in \cite{braga,grimmett,steif}. 

The resulting mathematical theory was quickly developed and became one of the main branches of contemporary probability. On the other hand, new percolation models appear in the
literature as an alternative to understand issues of biological and physical interest. Some instances are the $AB$ percolation, invasion percolation, oriented percolation, and continuum percolation models, among many others --- for more references see \cite{grimmett,roy}. As a theoretical tool, combined with coupling techniques, percolation theory is also useful in the construction and analysis of stochastic processes \cite{CS, durrett, galves, schinazi}.

The accessibility percolation model in trees was introduced in \cite{NK} inspired by questions from evolutionary biology. In that work an $n$-tree is considered and a continuous random variable $X_v$ is associated to each vertex $v$, independently of everything else. One of the main questions in this model is whether there exists a path of nearest neighbors $v_1,v_2,v_3\ldots$ such that
$$
X_{v_1}<X_{v_2}<X_{v_3}<\cdots
$$

\noindent with positive probability. This type of path is called {\it{accessible path}}. In \cite{NK} the authors derived an asymptotic result for the probability of having at least one accessible path connecting the root with the $kth$-level of an $n(k)$-tree, with $n(k):=\alpha\, k$, and $\alpha$ some arbitrary positive constant. Thus, they proved the existence of a percolation threshold for this model as $k\rightarrow \infty$. Indeed, they showed that the probability of having at least one accessible path goes to zero for $\alpha < \alpha_c$ and converges to a positive number for $\alpha > \alpha_c$ with $1/e \leq \alpha_c \leq 1$. Later, this result was complemented in  \cite{roberts} where the authors showed that this probability converges to 1 for $\alpha > 1/e$. Recently, a related problem was analyzed in the hypercube in \cite{BBZ,lili}.

In this work we establish some properties of the accessibility percolation model on spherically symmetric trees. It is a well known fact that the study of tree-indexed stochastic processes is of interest in understanding evolutionary biological questions. An instance of such a process is obtained by considering a phylogenetic or evolutionary tree, with vertices representing species and edges representing evolutionary relationships among them, and associating a (random) fitness value to each individual. Our model may be seen as a simple stochastic model for evolutionary trees. Different levels of the spherically symmetric tree represent different {\it{generations}} of species and the degree of each vertex represents the number of {\it{offspring}} species which appear in fixed intervals of time. In this sense, the varying environment of the tree is translated as a varying mutation rate. Related stochastic models for phylogenetic trees are proposed in \cite{fabio/herve/rinaldo,liggett/rinaldo}. In \cite{liggett/rinaldo} the authors propose a model which considers a birth and death component and a fitness component. Then, depending on the value of the (constant) mutation rate, they obtain a phylogenetic tree consistent with an influenza tree and also with an HIV tree.

 The paper is organized as follows. Section \ref{model} gathers the formal notations and definitions of the model and states some basic results. A discussion about the phase transition of the model is included in Section \ref{phase} and Section \ref{martingale} proposes a martingale approach to deal with this type of models. Indeed, such approach constitutes an alternative to the methods previously used in the literature to analyse the accessibility percolation model. Last section is devoted to a connection with the theory of records.


\section{The model and basic results}\label{model} 


\subsection{Trees}


We consider an infinite, locally finite, rooted tree $T = (\mathcal{V}, \mathcal{E})$. We denote the root of $T$ by ${\bf{0}}$. Here $ \mathcal{V} $ stands for the set of vertices and
$\mathcal{E} \subset \{\{u,v\}: u,v \in \mathcal{V}, u \neq v\}$ stands for the set of edges. If $\{u,v\}\in \mathcal{E}$, we say that $u$ and $v$ are neighbors, which is denoted by $u\sim v$. The degree of
a vertex $v$, denoted by $d(v)$, is the number of its neighbors. A path in $T$ is a finite sequence $ v_0, v_1, \dots, v_n $ of distinct vertices such that $ v_i \sim v_{i+1} $ for
each $ i $. Since $T$ is a tree, there is a unique path connecting any pair of distinct vertices $u$ and $v$. Therefore we may define the distance between them, which is denoted by
$d(u,v) $, as the number of edges in such path. For each $v\in \mathcal{V}$ define $|v|:=d({\bf{0}},v)$.

For $ u,v \in \mathcal{V} $, we say that $u\leq v$ if $u$ is one of the vertices of the path connecting ${\bf{0}}$ and $v$; $u<v$ if $u\leq v$ and $u\neq v$. 
We call $v$ a \textit{descendant} of $u$ if $u\leq v$ and denote by $T^u = \{v \in \mathcal{V}: u \leq v\}$ the set of descendants of $u$. On the other hand, $v$ is said to be a \textit{successor} of $u$ if $u\leq v$ and $u \sim v$. For $n\geq 1$, we denote by $\partial T_n$ the set of vertices at distance $n$ from the root. That is, $\partial T_n= \{v \in \mathcal{V}: |v|=n\}$.

In this work we deal with spherically symmetric trees (SST) which are trees where the degree of any vertex depends only on its distance from the root. In other words, for any 
$v \in \mathcal{V}$ we have $d(v)=f(|v|)$ where $f:=(f(i))_{i\geq 0}$ is a given sequence of natural numbers. The function $f$ is called the growth function of the tree. Such trees
will be denoted by $T_f$. We point out that there is no more information in $T_f$ than that contained in the sequence $(|\partial T_n |)_{n\geq 1}$.


\subsection{The accessibility percolation model} 


Let $T=(\mathcal{V}, \mathcal{E})$ be a SST. To each $v \in \mathcal{V}$ we associate a random fitness value. More precisely, let $\mathcal{X} =
\{X_v : v \in \mathcal{V}\}$ be a family of absolutely continuous, independent identically distributed (iid) and non-negative random variables assuming their values in a common set $S$. Denote by 
$\mathcal{Q}$ its common law defined in some $\sigma$-algebra $\mathcal{H}$. Now we formally describe the probability space $(\Omega, \mathcal{F},\mathbb{P})$ where all our analysis is
developed. Let $\Omega = S^{\mathcal{V}}$ be the space of ordered sequences of real numbers $\omega =(\omega_v)_{v \in \mathcal{V}}$ where $\omega_v \in S$ for each $v$. $\Omega$ is supplied with the corresponding product $\sigma$-algebra and denote it by $\mathcal{F}$. The probability measure $\mathbb{P}$ is the product measure in $(\Omega , \mathcal{F})$ with one-dimensional marginals given by $\mathcal{Q}$.

In order to formalize what percolation means in our context and state the results, we need some extra definitions. A path $v_0, v_1,\ldots,v_n$ in $T$ is said to be an accessible
path if $X_{v_0} < X_{v_1} < X_{v_2} < \cdots < X_{v_n}$. We denote by $\cam{v_0}{v_n}$ the event that $v_n$ is connected to $v_0$ through an accessible path (see Figure 1).

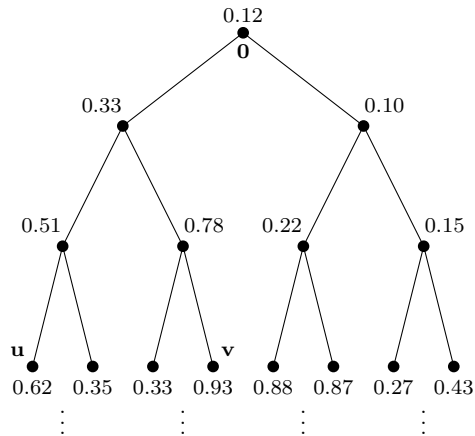
\begin{figure}[h]\label{ap}

\begin{center}
\begin{tikzpicture}[scale=0.8]


\draw [very thick] (0,-0.45) circle (2pt); 
\filldraw [black] (0,-0.45) circle (2pt); \draw (0,-0.15) node[font=\footnotesize] {$0.12$}; \draw (0,-0.75) node[font=\footnotesize] {${\bf 0}$};

\draw [very thick] (2,-2) circle (2pt); 
\filldraw [black] (2,-2) circle (2pt); \draw (2.35,-1.65) node[font=\footnotesize] {$0.10$};

\draw [very thick] (-2,-2) circle (2pt);
\filldraw [black] (-2,-2) circle (2pt); \draw (-2.35,-1.65) node[font=\footnotesize] {$0.33$};

\draw [very thick] (3,-4) circle (2pt);
\filldraw [black] (3,-4) circle (2pt); \draw (3.35,-3.65) node[font=\footnotesize] {$0.15$};
\draw [very thick] (1,-4) circle (2pt);
\filldraw [black] (1,-4) circle (2pt); \draw (0.65,-3.65) node[font=\footnotesize] {$0.22$};
\draw [very thick] (-3,-4) circle (2pt);
\filldraw [black] (-3,-4) circle (2pt); \draw (-3.35,-3.65) node[font=\footnotesize] {$0.51$};
\draw [very thick] (-1,-4) circle (2pt);
\filldraw [black] (-1,-4) circle (2pt); \draw (-0.65,-3.65) node[font=\footnotesize] {$0.78$};

\draw [very thick] (3.5,-6) circle (2pt);
\filldraw [black] (3.5,-6) circle (2pt); \draw (3.5,-6.35) node[font=\footnotesize] {$0.43$};
\draw [very thick] (1.5,-6) circle (2pt);
\filldraw [black] (1.5,-6) circle (2pt); \draw (1.5,-6.35) node[font=\footnotesize] {$0.87$};
\draw [very thick] (-3.5,-6) circle (2pt);
\filldraw [black] (-3.5,-6) circle (2pt); \draw (-3.5,-6.35) node[font=\footnotesize] {$0.62$}; \draw (-3.75,-5.75) node[font=\footnotesize] {${\bf u}$};
\draw [very thick] (-1.5,-6) circle (2pt);
\filldraw [black] (-1.5,-6) circle (2pt); \draw (-1.5,-6.35) node[font=\footnotesize] {$0.33$};
\draw [very thick] (2.5,-6) circle (2pt);
\filldraw [black] (2.5,-6) circle (2pt); \draw (2.5,-6.35) node[font=\footnotesize] {$0.27$};
\draw [very thick] (0.5,-6) circle (2pt);
\filldraw [black] (0.5,-6) circle (2pt); \draw (0.5,-6.35) node[font=\footnotesize] {$0.88$};
\draw [very thick] (-2.5,-6) circle (2pt);
\filldraw [black] (-2.5,-6) circle (2pt); \draw (-2.5,-6.35) node[font=\footnotesize] {$0.35$};
\draw [very thick] (-0.5,-6) circle (2pt);
\filldraw [black] (-0.5,-6) circle (2pt); \draw (-0.5,-6.35) node[font=\footnotesize] {$0.93$}; \draw (-0.25,-5.75) node[font=\footnotesize] {${\bf v}$};


\draw (0,-0.45) -- (2,-2) -- (3,-4) -- (3.5,-6);
\draw (0,-0.45) -- (-2,-2) -- (-3,-4) -- (-3.5,-6);
\draw (2,-2) -- (1,-4) -- (0.5,-6);
\draw (-2,-2) -- (-1,-4) -- (-0.5,-6);
\draw (1,-4) -- (1.5,-6);
\draw (-1,-4) -- (-1.5,-6);
\draw  (3,-4) -- (2.5,-6);
\draw (-3,-4) -- (-2.5,-6);


\draw (-3,-6.8) node[font=\footnotesize] {$\vdots$};
\draw (-1,-6.8) node[font=\footnotesize] {$\vdots$};
\draw (1,-6.8) node[font=\footnotesize] {$\vdots$};
\draw (3,-6.8) node[font=\footnotesize] {$\vdots$};

\end{tikzpicture}
\end{center}
\caption{A realization of the accessibility percolation model on a tree. 
In this example $u$ and $v$ are the only vertices, at distance $3$ from the root, which are connected to ${\bf 0}$ through an accessible path.}
\end{figure}

For each $n \in \mathbb{N}$, let $\Lambda_n$ be the event that $\partial T_n$ is accessible from the root. In other words,

\[
\Lambda_n := \Lambda_n (T)= \{ \cam{{\bf{0}}}{v}, \ \mbox{for some} \ v \in \partial T_n \}.
\]

In addition let $N_n$ be the number of accessible paths connecting the root with the $nth$-level of the tree, that is, $N_n := |\{v\in \partial T_n:\cam{{\bf{0}}}{v} \}|$.
 
We call this model the accessibility percolation model on $T$, and we denote it by $\mathcal{AP}(T,\mathcal{X})$.

\begin{definition}
We say that there is percolation in $\mathcal{AP}(T,\mathcal{X})$ if the event $\displaystyle \cap_{n \in \mathbb{N}} \Lambda_n$ occurs with positive probability.
\end{definition}

The main object of study is the percolation probability $\theta(T,\mathcal{X})$ given by
\[
\theta(T,\mathcal{X})=\mathbb{P}[\displaystyle \cap_{n \in \mathbb{N}} \Lambda_n].
\]
Since $\mathcal{X}$ is a family of iid random variables, the value of $\theta(T,\mathcal{X})$ does not
depend on the distribution of $\mathcal{X}$. Thus, we can safely remove $\mathcal{X}$ from the notation for $\theta$. Indeed, the probability of having an accessibility path of length $n$ equals $1/(n+1)\,!$

\begin{remark}\label{limprob}
Since $(\Lambda_n)_{n \in \mathbb{N}}$ forms a decreasing sequence of events, then $\theta(T):=\mathbb{P}(\displaystyle \cap_n \Lambda_n) = \displaystyle \lim_{n \rightarrow \infty}
\mathbb{P}(\Lambda_n)$.
\end{remark}

It is useful to compare, in some way, accessibility percolation models corresponding to different trees. Let $T_1 = (\mathcal{V}_1, \mathcal{E}_1)$ and $T_2= (\mathcal{V}_2, \mathcal{E}_2)$ be two trees with
common root ${\bf{0}}$. We say that $T_1$ is dominated by $T_2$ and write $T_1 \prec T_2$ if $\mathcal{V}_1 \subset \mathcal{V}_2$ and $\mathcal{E}_1 \subset \mathcal{E}_2$. If $T_1 \prec T_2$, then a simple coupling argument shows that every accessible path connecting ${\bf{0}}$ to $\partial T_{1_n}$ is also an accessible path connecting ${\bf{0}}$ to $\partial T_{2_n}$. This, in turns, implies that $\theta(T)$ is a non-decreasing function in $T$.  Coupling is a well known technique in percolation theory; we refer the reader to \cite[Chapter 3]{grimmett2} for an example of application in the bond percolation model. 

As a consequence of the previous remarks, the following result holds.

\begin{lemma}\label{lema}
If $T_1 \prec T_2$, then $\theta(T_1)\leq \theta(T_2)$.
\end{lemma}

Lemma \ref{lema} above is general enough for our purposes. However, it is possible to obtain a similar result for a more general comparison between accessibility percolation models related to different trees. Let $T_1 = (\mathcal{V}_1, \mathcal{E}_1)$ and $T_2= (\mathcal{V}_2, \mathcal{E}_2)$ be two trees with roots ${\bf{0}}_1$ and ${\bf{0}}_2$, respectively. We say that $T_1$ is comparable to $T_2$ and write $T_1 \prec_s T_2$ if there is an isomorphism between $T_1$ and a subtree $T_2^v$ of $T_2$ rooted at $v\in \mathcal{V}_2$. Note that  $T_2^v$ belongs to the subtree of $T_2$ formed by the descendants of $v$. Then, a simple coupling argument allows us to conclude the following result. 

\begin{lemma}\label{lema2}
If $T_1 \prec_s T_2$, then $\theta(T_1)>0$ implies $\theta(T_2)>0$.
\end{lemma}

\section{Phase transition}\label{phase}
Let $T_{!}$ denotes the factorial tree whose growth function is given by $f(i)=i+1$, $i\geq 0$. Note that 
$$
\mathbb{P}(\Lambda_{n})\leq \frac{|\partial T_{!,n}|}{(n+1)!} = \frac{1}{n+1}.
$$
By letting $n \rightarrow \infty$ (see Remark \ref{limprob}), we have $\theta(T_{!})=0$. Therefore, Lemma \ref{lema} shows the necessity of considering trees with a superlinear growth function in order to establish the existence of percolation. This behavior is not surprising since the accessibility percolation model {\it{penalizes}} long accessibility paths. Thus, we turn our attention to SST $T_{\alpha}$ with growth function given by $f(i)=\lceil(i+1)^\alpha\rceil$, where $\alpha>0$ is a constant. By Lemma \ref{lema}, $\theta(\alpha):=\theta(T_{\alpha})$ is non-decreasing in $\alpha$ and the critical parameter $\alpha_c :=\inf\{\alpha : \theta(\alpha)>0\}$ is well defined. The next result localizes the exact point where phase transition occurs for the accessibility percolation model on $T_{\alpha}$. 
 
\begin{theorem}\label{teo2}
Consider the $\mathcal{AP}(T_{\alpha},\mathcal{X})$ model. Then, the critical parameter is given by $\alpha_c =1$. More precisely, $\theta(\alpha)=0$ if $\alpha \leq 1$ and $\theta(\alpha)>0$ if $\alpha >1$.
\end{theorem}

\begin{proof}
The first statement of Theorem \ref{teo2}, i.e. the subcritical regime, is a consequence of Lemma \ref{lema} and the fact that $\theta(T_{!})=0$. The second part, i.e. the supercritical regime, may be obtained as a consequence of Theorem 3.3 in \cite{BRZ}. In \cite{BRZ} the authors give sufficient conditions under which a Galton-Watson process in varying environment with selection (GWVES) survives. The Galton-Watson process in varying environment is a branching process where the offspring distribution may differ from generation to generation. The selection procedure assumed in \cite{BRZ} is the following: each new born individual has associated a fitness selected from a i.i.d. family of continuous time random variables, and each individual survives only when its fitness is greater than its parent's fitness. Thus defined, there is a one-to-one correspondence between GWVES and accessibility percolation models on random trees. In our case, consider first a branching process assuming that each individual from the $ith$-generation has, before selection, $\lceil(i+1)^\alpha\rceil$ offspring at generation $i+1$, for any $i=0,1,2,\ldots$ Assuming that selection is performed according to i.i.d. uniform random variables, we have that a realization of this GWVES process coincides with a realization of the $\mathcal{AP}(T_{\alpha},\mathcal{X})$ model. By Theorem 3.3 in \cite{BRZ} we know that this GWVES survives provided
$$\sum_{i=1}^{\infty}\frac{1}{\lceil(i+1)^\alpha\rceil}<\infty,$$
which occurs whenever $\alpha>1$. The main ideia behind the obtention of the previous condition is the interpretation of the GWVES as a particular case of a branching random walk (BRW). In this case, the fitness of each particle in the GWVES is interpreted as the position of the respective particle in the BRW. We refer the reader to \cite{BRZ} for a more detailed explanation of such a comparison and a review of sufficient conditions for survival of BRWs.
\end{proof}

\begin{remark}\label{remarkop1}
Theorem \ref{teo2} concerns the accessibility percolation model on a special class of spherically symmetric trees. It still is an open problem to find sufficient and necessary conditions, for any growth function $(f(i))_{i\geq 0}$, under which there is, or not, percolation. In this direction, Theorem 3.3 in \cite{BRZ} may be understood as saying that a sufficient condition for the existence of percolation, on a tree with growth function given by $(f(i))_{i\geq 0}$, is 
$$\sum_{i=0}^{\infty}\frac{1}{f(i)}<\infty.$$  
On the other hand, since $\mathbb{P}(\Lambda_{n})\leq |\partial T_{n}|/(n+1)!$ we have that a sufficient condition for the absence of percolation, on the same tree, is
\[
\liminf_{n\to \infty}\frac{\prod_{i=0}^{n-1}f(i)}{(n+1)!}=0.\]  
Notice that the previous conditions do not apply to the case $f(i)=\lceil \alpha (i+1)\rceil$, for $i\geq 0$, and $\alpha >1$.
\end{remark}

\begin{remark}
We also observe that the second moment method applied in the model of \cite{NK,roberts} for controlling the number of accessible paths, $N_n$, connecting the root with the $nth$-level of the tree does not work here. This is because, in our case, $\mathbb{E}(N_n^2)$ grows much faster than $\mathbb{E}(N_n)^2$. 
\end{remark}

\begin{theorem}
Consider the $\mathcal{AP}(T_{\alpha},\mathcal{X})$ model. If $\alpha < 1$, then there is no percolation in $T_{\alpha}^v$ for any $v \in T_{\alpha}$. If
$\alpha > 1$, then the event $A=\{$there is percolation starting from $v \in T_{\alpha}$ i.o.$\}$ has probability one. 
\end{theorem}

\begin{proof}
Suppose that $\alpha < 1$. Consider a vertex $v \in \partial T_{\alpha,i}$. Let us show that the probability of existing an infinite accessible path starting from $v$ is zero. For each $n \in \mathbb{N}$, let $\Lambda_n^{v}$ be the event that $\partial T_{\alpha, n}^v$ is accessible from $v$. 
More explicitly, 
\begin{align*}
\Lambda_n^{v} := \{ \cam{{v}}{u}, \ \mbox{for some} \ u \in \partial T_{\alpha, n}^v \}.
\end{align*}
Now choose $\varepsilon \in (0, 1)$ and $n_{\varepsilon}$ such that $n \geq n_{\varepsilon}$ implies
\begin{align*}
\frac{(i+n)^\alpha + 1}{n} < \varepsilon . 
\end{align*}
Then for $n \geq n_{\varepsilon}$
\begin{align*}
\mathbb{P}(\Lambda_{n}^{v}) & \leq \frac{\prod_{j=1}^{n}\lceil (i + j)^\alpha\rceil}{(n+1)!} \\
& \leq \frac{\prod_{j=1}^{n_{\varepsilon}}\lceil (i + j)^\alpha\rceil}{n_{\varepsilon}!} \varepsilon^{n-n_{\varepsilon}+1} 
\to 0 \; \text{ as } \; n \to \infty,
\end{align*}
that is, there is no percolation starting from $v$ with probability one. Since $T_{\alpha}$ has a countable 
number of vertices the first claim follows.

Now assume that $\alpha > 1$. We will prove the second claim by construction. 
Consider a vertex $v_2 \in \partial T_{\alpha,2}$ and define $\Lambda^{v_2}$ to be the event there is accessibility  
percolation in $T_{\alpha}^{v_2}$. Since the sequence $\left(\lceil\left(i+1\right)^{\alpha}\rceil\right)_i$ is non-decreasing, a simple coupling argument (see Lemma 2.2) yields
$\mathbb{P}(\Lambda^{v_2}) \geq \theta(\alpha) > 0$. Then take a vertex $v_3 \in \partial T_{\alpha,3}$ which is not a successor of $v_2$, and define the event $\Lambda^{v_3}$ in a similar way. Again $\mathbb{P}(\Lambda^{v_3}) \geq \theta(\alpha)$. Observe that $\Lambda^{v_2}$ and $\Lambda^{v_3}$ 
are independent since they are defined in terms of two disjoint families of independent random variables.

Suppose we have defined $v_2, v_3, \ldots, v_n$ and suppose we have defined the events $\Lambda^{v_2}, \Lambda^{v_3}, \ldots, \Lambda^{v_n}$; then we define inductively $v_{n+1} \in \partial T_{\alpha, n+1}$ in such a way it is not a successor of $v_2, \ldots, v_n$. Then $\mathbb{P}(\Lambda^{v_{n+1}}) \geq \theta(\alpha)$ and the events 
$\Lambda^{v_2}, \Lambda^{v_3}, \ldots , \Lambda^{v_{n+1}}$ are independent. Then, $\{ \Lambda^{v_n} \}_{n \geq 2}$ is a sequence of independent 
events such that 
\begin{align*}
\sum_{n=2}^{\infty} \mathbb{P}(\Lambda^{v_n}) = \infty.
\end{align*}
It follows from the Borel-Cantelli lemma that $\mathbb{P}(A)=1$.

\end{proof}

\begin{proposition}\label{rightcont}
The percolation function $\theta$ is right continuous on $[0, \infty) \setminus \{ 1, 2, 3 , \ldots \}$.
\end{proposition}
\begin{proof}
Let
\begin{align*}
\theta_{n}(\alpha) := \mathbb{P}(N_n \geq 1) \; \text{ for } n \geq 1.
\end{align*}
We observe that $\theta_n$ is a continuous function at $\alpha \in [0, \infty) \setminus \{ 1, 2, 3 , \ldots \}$, 
since the event $\{ N_n \geq 1 \}$ depends on the vertices at distance $\leq n $ from the root: if we choose 
$\alpha_0$ sufficiently close to $\alpha$ so that $\lceil (1+i)^{\alpha_0} \rceil = \lceil (1+i)^{\alpha} \rceil$ 
for all $i \in \{0, 1, \ldots , n-1 \}$, then $\partial T_{\alpha, i} = \partial T_{\alpha_0, i}$ for all $i \in \{1, \ldots , n  \}$ and $\theta_n(\alpha) = \theta_n(\alpha_0)$. Furthermore, $\theta_n$ is non-increasing at 
$n$ and 
\begin{align*}
\theta(\alpha) = \inf_n\theta_{n}(\alpha) =  \lim_{n \to \infty} \theta_n(\alpha),
\end{align*}
which implies that $\theta$ is upper semicontinuous. Note next that a non-decreasing upper 
semicontinuous function is a right continuous function.
\end{proof}

\begin{remark}
In our model, in contrast to what happens in the model considered in \cite{NK,roberts}, the percolation 
probability in the supercritical regime is strictly less than $1$. In order to see that, call $u$ the unique 
neighbor of root ${\bf 0}$ and notice that $u$ has $\lceil 2^{\alpha} \rceil$ neighbors at distance 2 from the root. 
Call these neighbors $v_1, \ldots, v_{\lceil 2^\alpha \rceil}$. If the event
$\left\{X_{\o} > X_u \right\} \cup  \left\{ X_{u} > \max\{X_{v_1}, \ldots , X_{v_{\lceil 2^\alpha \rceil}} \} \right\}$ 
occurs then there is absence of percolation. Therefore,
\begin{align*}
\theta(\alpha) \leq 1 - \mathbb{P}(\left\{X_{\o} > X_u \right\} \cup X_{u} > \max\{X_{v_1}, \ldots , X_{v_{\lceil 2^\alpha \rceil}} \}) 
= \frac{\lceil 2^\alpha \rceil}{2(2 + \lceil 2^\alpha \rceil)},
\end{align*}
which goes to $1/2$ as $\alpha \to \infty$.
\end{remark}

\section{Records}\label{records}
The accessibility percolation model was motivated by biological questions. However, this model is also related to record theory. In order to compare both models we assume, without loss of generality, that accessibility paths are defined by considering random variables in decreasing order. Note that this does not change our results. Now consider a sequence of competitions such that on the $n$-th edition there are $\alpha(n)$ participants where $\alpha(1)=1$. Assume that the scores obtained in the different editions are given by a sequence $Y_n=(Y_1^n, \ldots , Y_{\alpha(n)}^n)$, $n=1,2,\ldots$, of iid random vectors composed by absolutely continuous, iid random variables, each having continuous cumulative distribution function (cdf) $F$. Then, the score of the winner of the $n$-th edition is given by $X_n = \max \{Y_1^n, \ldots , Y_{\alpha(n)}^n \}$ and its cdf  by $F_n(x) = F(x)^{\alpha(n)}$. So defined, this is the well known $F^{\alpha}$ record model, also referred  as the $F^{\alpha}$ scheme in \cite{nevzorov}. See Chapter 6, \cite{arnold}, for a review of results of $F^{\alpha}$ record models.

Now we define record time sequence $(S_n)_{n\geq 1}$ as follows. Let $S_0 = 1$ and, for $n \geq 1$, let $S_n = \inf \{j : X_j > X_{S_{n-1}}\}.$ The (asymptotic) distribution of these random variables, as well as the one of the inter-records $S_n - S_{n-1}$, are well studied in record theory. A different question that arises is if it is possible to have a record at each edition of the competition with positive probability. In other words, is it possible to have $\mathbb{P}\left(\cap_{n=1}^{\infty}\{S_n = n\}\right)> 0?$ Theorem \ref{teo2} is related to the answer of this question provided $\alpha(n)= \lceil(n+1)^\alpha\rceil$. Indeed, by a coupling argument, the event of observing a record at each edition of the competition implies the event of percolation in our model, which occurs if, and only if, $\alpha >1$. 

\section{Discussion}\label{martingale}

In this work we establish new properties of the accessibility percolation model on trees. Nevertheless, some open problems remain. One open question is related to Remark \ref{remarkop1}. If the accessibility percolation model is defined in a generic spherically symmetric tree, then it would be interesting to prove that there is phase transition and that the corresponding critical parameter is a function of the growth function of the underlying tree. In particular, as we point out in Remark \ref{remarkop1}, there is a entire family of trees where neither the second moment method used in \cite{NK}, nor the techniques used in this work apply. Therefore new methods must be explored, possibly supported by numerical simulations, to solve this question.

On the other hand, once we know that there is phase transition on a given tree, it would be interesting to describe the asymptotic behavior of the number of accessible paths. In this work, we consider the $\mathcal{AP}(T_{\alpha},\mathcal{X})$. If $N_{n}$ be the number of accessible paths connecting the root to $\partial T_{\alpha,n}$ then, by Theorem \ref{teo2}, we have
\begin{align*}
\lim_{n\to \infty}\mathbb{P}(N_n \geq 1) \left\{
\begin{array}{ll}
=0,& \text{if }\alpha \leq 1 \\
>0, &\text{if }\alpha > 1.
\end{array}\right.
\end{align*}

One possible approach is to study the asymptotic behavior of the random variable $N_n$ through martingales. Notice that
$$
N_{n} = \sum_{i=1}^{|\partial T_{\alpha,n}|}I_{i},
$$
where $\left\{ I_{i} \right\}_{i}$ is a sequence of Bernoulli r.v.'s which are not independent with parameter $1/n\,!$ Set 
\begin{align}\label{beta}
\alpha_{n} := \dfrac{\lceil n^\alpha\rceil}{n+1} \; \text{ and } \;
\beta_n := \prod_{i=1}^{n}\alpha_i \text{ for } n \geq 1,
\end{align}
and note that
\begin{align*}
\mathbb{E}(N_{n}) = \frac{|\partial T_{\alpha,n}|}{(n+1)!} = \prod_{i = 1}^{n} \alpha_i = \beta_n.
\end{align*}
Put $M_{0} =1$ and define 
\begin{align*}
M_{n} := \frac{N_{n}}{\beta_n} \text{ for } n \geq 1. 
\end{align*}
A straightforward computation shows that $\{ M_{n} \}_{n \geq 0}$ is a non-negative martingale with respect to the filtration $\left\{ \mathcal{F}_n \right\}_{n \geq 1}$ given by
$\mathcal{F}_n = \sigma(X_v; v \in T_{\alpha,n})$. It follows from the theory of martingales that the sequence $\left(M_n\right)_n$ converges a.s. to a non-negative random variable $M$ with $\mathbb{E}(M) \leq 1$. Since $\mathbb{E}(M_{n}^{2})$ is not uniformly bounded, we can not use the $\mathcal{L}^{2}$ martingale convergence theorem to conclude that $\mathbb{P}(M > 0) > 0$; however if we manage to prove it, then we may conclude that for any realization $\omega$ of the model where $M(\omega) >0$
\[
N_n(\omega) \sim \beta_n M(\omega).
\]
Here $f(n) \sim g(n)$ stands for $\lim_{n \rightarrow \infty} f(n)/g(n)=1$.

Another question to be mentioned concerns the continuity of the percolation function on $\left[\alpha_c,+\infty \right)$. Notice that we only prove, see Proposition \ref{rightcont}, that this function is right continuous on $[0, \infty) \setminus \{ 1, 2, 3 , \ldots \}$.  This could be proved, as in other percolation models, by proving that there is just one unbounded connected component.

\section*{Acknowledgements}
This work was supported by FAPESP [grant numbers 2015/20110-0, 2013/03898-8, 2016/11648-0]; and CNPq [grant numbers 461365/2014-6, 479313/2012-1, 304676/2016-0]. We are grateful to Rahul Roy for pointing out to us the relation between records and our model. Thanks are also due to Francisco Rodrigues and Rafael Grisi for fruitful discussions at different stages of the current work, and to the reviewers for their corrections and remarks which helped to improve the presentation of our work. 


\end{document}